\documentclass{article}
\usepackage{structure}

\newcommand{\preclex}{\preceq_{\text{lex}}}
\renewcommand{\vec}[1]{\mathbf{#1}}
\newcommand{\PA}{\mathcal{PA}}

\title{Discrepancy of Arithmetic Progressions in Boxes and Convex Bodies}
\author{Lily Li\thanks{University of Toronto, Department of Computer Science (xinyuan@cs.toronto.edu).} \ and Aleksandar Nikolov\thanks{University of Toronto, Department of Computer Science (anikolov@cs.toronto.edu).}}
\date{}

\begin{document}
\maketitle
\begin{abstract}
    The combinatorial discrepancy of arithmetic progressions inside $[N] := \{1, \ldots, N\}$ is the smallest integer $D$ for which $[N]$ can be colored with two colors so that any arithmetic progression in $[N]$ contains at most $D$ more elements from one color class than the other. Bounding the discrepancy of such set systems is a classical problem in discrepancy theory. More recently, this problem was generalized to arithmetic progressions in grids like $[N]^d$ (Valk{\'o}) and $[N_1]\times \ldots \times [N_d]$ (Fox, Xu, and Zhou). In the latter setting, Fox, Xu, and Zhou gave upper and lower bounds on the discrepancy that match within a $\frac{\log |\Omega|}{\log \log |\Omega|}$ factor, where $\Omega := [N_1]\times \ldots \times [N_d]$ is the ground set. In this work, we use the connection between factorization norms and discrepancy to improve their upper bound to be within a $\sqrt{\log|\Omega|}$ factor from the lower bound. We also generalize Fox, Xu, and Zhou's lower bound, and our upper bounds to arithmetic progressions in arbitrary convex bodies.
    %
\end{abstract}
\section{Introduction}
Combinatorial discrepancy theory studies set systems (or, equivalently, hypergraphs) $(\Omega, \mathcal{S})$ where $\Omega$ is the underlying (finite) universe of elements and $\mathcal{S}$ is a family of subsets of $\Omega$. For a fixed coloring $\chi: \Omega \rightarrow \{-1, 1\}$, the discrepancy of $\mathcal{S}$ with respect to $\chi$, denoted $\disc(\mathcal{S}, \chi)$, is defined as 
\begin{equation}
    \disc\left(\mathcal{S}, \chi\right) \coloneqq \max_{S \in \SS}\left|\sum_{\omega \in S}\chi(i)\right|.
\end{equation}
Then the discrepancy of $\mathcal{S}$, denoted by $\disc(\mathcal{S})$, is defined as
\begin{equation}\label{eq:disc}
    \disc\left(\mathcal{S}\right) \coloneqq \min_{\chi: \Omega \to \{-1, 1\}} \disc\left(\mathcal{S}, \chi\right)= \min_{\chi: \Omega \to \{-1, 1\}}\max_{S \in \SS}\left|\sum_{\omega \in S}\chi(i)\right|.
\end{equation}
Thus, a set system has low discrepancy if it is possible to color the elements in its universe with two colors so that all the sets in the system are simultaneously balanced. Combinatorial discrepancy has numerous connections to other fields of mathematics and computer science, e.g., to irregularities of distribution~\cite{BeckChen}, approximation algorithms for $\mathsf{NP}$-hard problems~\cite{Rothvoss-binpacking,HR17-binpacking,BRS22-flowtime}, data structures~\cite{Larsen-range}, among others. We refer to the books~\cite{matousek2009geometric,Chazelle-book}, and the thesis~\cite{sasho-thesis} for further references.

One of the first set systems whose discrepancy was thoroughly studied is the set of all arithmetic progressions (APs) on $[N] \eqdef \{1, \ldots, N\}$. We denote this set system $\left([N], \AA_N\right)$ where, $\AA_N \eqdef \{\AP(a,b,\ell) \cap [N]: a,b \in [N], \ell \in \mathbb{Z}_{>0}\}$, and $\AP(a,b,\ell)$ is the arithmetic progression starting at $a$ with difference step sizes $b$ and $\ell$ steps:
\begin{equation}\label{eq:ap}
    \mathrm{AP}(a, b, \ell) =  \{a + i b: i = 0, \ldots, \ell-1\}.
\end{equation} 
In an important early paper~\cite{roth1964remark}, Roth used a Fourier theoretic argument to show that
\begin{equation}\label{eq:rothlb}
    \disc(\AA) \gtrsim N^{1/4},
\end{equation}
where the notation $A \gtrsim B$ means that there exists an absolute constant $C$ such that $A \ge CB$. We will similarly use $A \lesssim B$ for $B \gtrsim A$, and $A \asymp B$ for $A \lesssim B \lesssim A$. Roth remarked that his lower bound shows that ``there are various limitations to the extent to which a sequence of natural numbers can be well-distributed simultaneously among and within all congruence classes.'' Indeed, since $[N]$ is itself an arithmetic progression, a low discrepancy coloring must color about half of $[N]$ with $+1$. Taking $s$ to be the sequence of all integers in $[N]$ colored $+1$, Roth's lower bound~\eqref{eq:rothlb} shows that the density of $s$ in some arithmetic progression must be substantially different from $\frac12$. Thus $s$ cannot be simultaneously well distributed within all arithmetic progressions. This fact is of similar nature as, for example, van der Waerden's theorem.

A uniformly random coloring of $[N]$ achieves discrepancy in $O(\sqrt{N\log N})$ with high probability, so Roth conjectured that this simple upper bound ought to be closer to the truth than his lower bound~\eqref{eq:rothlb}. This turned out to be false, and, after several weaker results, Beck showed that $\disc(\AA) \lesssim N^{1/4}\left(\log N\right)^{5/2}$, which matches~\eqref{eq:rothlb} up to the logarithmic factors~\cite{beck1981roth}. Beck introduced his celebrated partial coloring method in the proof of this upper bound. A subsequent refinement of the partial coloring method was used by Matou\v{s}ek and Spencer to show that
\begin{equation}\label{eq:matousekub}
   \disc(\AA) \lesssim N^{1/4} 
\end{equation} 
proving that Roth's lower bound was tight up to constant factors~\cite{matouvsek1996discrepancy}.  

Since then, Valk{\'o}~\cite{valko2002discrepancy} and Fox, Xu, Zhou~\cite{fox2024discrepancy} have investigated upper and lower bounds on the discrepancy of arithmetic progressions in $d$-dimensional axis-aligned boxes. We define the set system of $d$-dimensional arithmetic progressions following the notation of Fox, Xu, Zhou~\cite{fox2024discrepancy}. Here, the universe is the set of all integer points inside an axis-aligned, $d$-dimensional box. In particular, given positive integers $N_1, ..., N_d$, define the vector $\mathbf{N} \coloneqq (N_1, \ldots, N_d)$, as the side lengths of the box in each dimension. Then we have $\Omega_{\mathbf{N}} \eqdef [N_1] \times \cdots \times [N_d]$. Note that the size of $\Omega_{\mathbf{N}}$ is $|\Omega_{\mathbf{N}}| = N_1 \cdots N_d$. The set system of arithmetic progressions in $\Omega_{\vec{N}}$, denoted $\AA_{\mathbf{N}}$, contains all sets of the form $\Omega_{\mathbf{N}} \cap \mathrm{AP}_d(\mathbf{a}, \mathbf{b},\ell)$, where 
\begin{equation}\label{eq:apd}
    \mathrm{AP}_d(\mathbf{a}, \mathbf{b},\ell) \eqdef \{\mathbf{a} + i\mathbf{b}: i = 0, \ldots, \ell-1\},
\end{equation}
and $\mathbf{a}$ and $\mathbf{b}$ are vectors in $\Z^d$ and $\ell$ is a positive integer. This matches the one-dimensional definition given in equation~\eqref{eq:ap}. 

Next, we define a function which bounds the discrepancy of the set system $(\Omega_{\mathbf{N}}, \AA_{\mathbf{N}})$. Let $f: \Z^d \rightarrow \R$ be
\begin{equation}\label{eq:ub}
    f\left(\mathbf{N}\right) \eqdef \max_{I \subseteq [d]} \left(\prod_{i \in I} N_i\right)^{\frac{1}{2|I| + 2}},
\end{equation}
where the empty product is equal to $1$ by definition. When $d = 1$, we have $f(N_1) = N_1^{1/4}$ which is asymptotically the upper and lower bound on $\AA_N$ as we have seen in equations~\eqref{eq:rothlb} and~\eqref{eq:matousekub}.

Valk\'{o}~\cite{valko2002discrepancy} considered the cube $\Omega = [N]^d$ and showed that 
\[f(\textbf{N}) = N^{\frac{d}{2d + 2}} \lesssim_d \disc\left(\AA_{\mathbf{N}}\right) \lesssim_d N^{\frac{d}{2d + 2}}\left(\log N\right)^{5/2}.\] 
Here, the notation $A \lesssim_d B$ means that there exists a constant $C_d$ that only depends on the dimension $d$ such that $A \le C_d B$. Analogously, $A\gtrsim_d B$ is equivalent to $B\lesssim_d A$, and $A \asymp_d B$ is equivalently to $A \lesssim_d B\lesssim_d A$.

Later, Fox, Xu, Zhou~\cite{fox2024discrepancy} removed the logarithmic term from Valk\'{o}'s upper bound, and considered general $d$-dimensional boxes, showing the following bounds on $\disc\left(\AA_{\mathbf{N}}\right)$.
\begin{theorem}{\cite[Theorem 1.3]{fox2024discrepancy}}\label{thm:FXZ}
    For any positive integer $d$ and vector $\mathbf{N} = \left(N_1, \ldots, N_d\right)$, where each $N_i$ is an integer greater than $1$, let $\AA_{\mathbf{N}}$ be the family of all arithmetic progression defined on the universe $\Omega_{\mathbf{N}} $. Then the discrepancy of $\AA_{\mathbf{N}}$ is bound by
    \begin{equation}\label{eq:FXZ}
        f\left(\mathbf{N}\right) \lesssim_d
        \disc\left(\AA_{\mathbf{N}}\right) \lesssim_d
        \frac{\log |\Omega_{\mathbf{N}}|}{\log \log |\Omega_{\mathbf{N}}|} \cdot  f\left(\mathbf{N}\right).
    \end{equation}
\end{theorem}
Note that a weaker upper bound of $\left(\log |\Omega_{\mathbf{N}}|\right)\cdot f\left(\mathbf{N}\right)$ can be achieved using standard partial coloring methods~\cite{matousek2009geometric}. The improvement by a $\log \log |\Omega_{\mathbf{N}}|$ factor relies on Fox, Xu, Zhou's upper bound for $[N]^d$, which itself relies on a sophisticated argument involving techniques from the geometry of numbers.

In this work, one of our main results, \Cref{thm:main}, is a tighter upper bound on the discrepancy of arithmetic progressions in $d$-dimensional boxes. 
\begin{restatable}[APs in Boxes]{theorem}{main}
\label{thm:main}%
    For any positive integer $d$ and $\mathbf{N} = \left(N_1, \ldots, N_d\right)$, where each $N_i$ is an integer greater than $1$, let $\AA_{\mathbf{N}}$ be the family of all arithmetic progression defined on the universe $\Omega_{\mathbf{N}} $. The discrepancy of $\AA_{\mathbf{N}}$ is bounded above as 
    \begin{equation}\label{eq:main}
        \disc\left(\AA_{\mathbf{N}}\right) \lesssim_d
        \sqrt{\log |\Omega_{\mathbf{N}}|} \cdot f\left(\mathbf{N}\right).
    \end{equation}
    Further, the upper bound is constructive, i.e., there exists a randomized algorithm which computes a coloring $\chi$ in time polynomial in $|\Omega_{\mathbf{N}}|$ such that, with  high probability, $\disc\left(\AA_{\mathbf{N}}, \chi\right) \lesssim_d \sqrt{\log |\Omega_{\mathbf{N}}|}\cdot f(\mathbf{N})$.
\end{restatable}
Unlike Theorem~\ref{thm:FXZ}, Theorem~\ref{thm:main} uses the connection between discrepancy and the $\gamma_2$ factorization norm due to Matou\v{s}ek, Nikolov, and Talwar~\cite{matouvsek2020factorization}, rather than the partial coloring technique. 
Instead of the delicate number theoretic arguments used by Fox, Xu, and Zhou, our proof uses only simple combinatorial arguments, and properties of factorization norms.

As had Fox, Xu, Zhou, we conjecture that the lower bound in equation~\eqref{eq:FXZ} is tight. Unfortunately, there seems to be no way to remove the final $\sqrt{\log |\Omega_{\mathbf{N}}|}$ factor using our methods. It is conceivable that partial coloring arguments --- even more sophisticated than the one used by Fox, Xu, Zhou~\cite{fox2024discrepancy} --- could give a tight bound.

In addition to arithmetic progressions in $d$-dimension axis-aligned boxes, we can generalize our upper  bounds to arithmetic progressions in \emph{any} $d$-dimensional convex body. First, we will need to extend our definition of $f$ from equation~\eqref{eq:ub} in order to encapsulate these new structures. Let $K$ be a $d$-dimensional convex body (i.e., a closed bounded convex set with non-empty interior) and $t$ a scalar.  Recall that $K-K$, i.e., the Minkowski sum of $K$ and $-K$, is defined as $\{\mathbf{x}-\mathbf{y}: \mathbf{x}, \mathbf{y} \in K\}$. Then define
\begin{equation}\label{eq:ub-general}
    \zeta_{K-K}(t) \eqdef |\Z^d \cap t(K-K)|. 
\end{equation}
to be the number of integer points inside $t(K-K)$. Further define 
\begin{equation}\label{eq:ub-body}
    f(K) \coloneqq \sqrt{s^*}\mbox{ where }s^* = \inf\left\{s: s \ge \zeta_{K-K}\left(\frac{1}{s}\right)\right\}.
\end{equation}
Note that when we restrict $K$ to be the $d$-dimensional box $[1, N_1]\times\cdots\times [1, N_d]$, we have $K-K=[-N_1+1, N_1-1] \times \cdots \times [-N_d+1,N_d-1]$, and thus
\[
\zeta_{K-K}(t) = 
\prod_{i=1}^d \left(1 + 2\floor{t (N_i-1)}\right)
\asymp_d \max_{I\subseteq [d]}\left(t^{|I|} \prod_{i \in I} N_i\right),
\]
where $A\asymp_d B$ is shorthand for  $A \lesssim_d B$ and $B\lesssim_d A$ holding simultaneously.
In the proof of Lemma~\ref{lem:gamma2-map-ub} below 
we show that $s=f(\mathbf{N})^2$ satisfies
\(
s \asymp \max_{I\subseteq [d]}\left(\frac{1}{s^{|I|}} \prod_{i \in I} N_i\right),
\)
and thus $f(K) \asymp_d f(\mathbf{N})$ in the special case where $K$ is an axis-aligned box. 

Similarly to~\Cref{thm:main}, \Cref{thm:generalbodiesthm} below bounds the discrepancy of arithmetic sequences in general $d$-dimensional convex bodies. 
\begin{restatable}[APs in Convex Bodies]{theorem}{generalbodiesthm}
\label{thm:generalbodiesthm}%
    Let $K$ be any $d$-dimensional convex body, $\mathbf{v}\in \R^d$ be a vector, and let $\Omega_{\mathbf{v}+K}$ be all the integer points in $\mathbf{v}+K$. If $\AA_{\mathbf{v} + K}$ is  the family of all arithmetic progression defined on the universe $\Omega_{\mathbf{v} + K}$, then
    \begin{equation}\label{eq:generalbodies-lb}
        f(K) \lesssim_d \sup_{\mathbf{v}\in \R^d} \disc\left(\AA_{\mathbf{v} + K}\right)
    \end{equation}
    and, for every $\mathbf{v} \in \mathbb{R}^d$, 
    \begin{equation}\label{eq:generalbodies-ub}
        \disc\left(\AA_{\mathbf{v} + K}\right) \lesssim_d \sqrt{\log (1+|\Omega_{\mathbf{v} + K}|)}\cdot f\left(K\right).
    \end{equation}
    where $f$ is defined in equation~\eqref{eq:ub-body}. 
    
    Moreover, if $K=-K$, i.e., $K$ is symmetric around the origin, then for $\AA_K \eqdef \AA_{\vec{0} + K}$,
    \begin{equation}\label{eq:generalbodies-sym}
        f(K) \lesssim_d \disc\left(\AA_{K}\right) \lesssim_d
        \sqrt{\log (1+|\Omega_K|)}\cdot f\left(K\right).
    \end{equation}

\end{restatable}

We note here that our proof of  \Cref{thm:generalbodiesthm} is not constructive, in the sense that we cannot, at present, guarantee an algorithm that finds a coloring achieving the claimed bound in time polynomial in the number of integer points in $K$. In addition to technical issues, like how the convex body $K$ would be given to an algorithm, and how the integer points inside it can be enumerated, this is also due to our reliance on a result of Banaszczyk~\cite{bana12} that presently has no algorithmic counterpart. This is the same reason why the best known bound for the discrepancy of axis-aligned boxes~\cite{nikolov-boxes}, and the best bounds for the Steinitz problem~\cite{bana12} are not constructive. 

Theorem~\ref{thm:generalbodiesthm} shows how the discrepancy of arithmetic progressions inside a convex body $K$ can be derived from the geometric properties of $K$. To illustrate its use, we state a simple corollary that determines how $\disc(\AA_{rK})$ grows as $r\to \infty$. 

\begin{corollary}
    Let $K$ be a $d$-dimensional convex body of volume $1$. Define $\AA_{\vec{v} + K}$ as in Theorem~\ref{thm:generalbodiesthm}. Then,
    \[
    \limsup_{r\to\infty} \frac{\sup_{\vec{v} \in \R^d} \disc(\AA_{\vec{v}+rK})}{r^{\frac{d}{2(d+1)}}\sqrt{\log r}} \lesssim_d 1; \hspace{2em} \liminf_{r\to\infty} \frac{\sup_{\vec{v} \in \R^d} \disc(\AA_{\vec{v}+rK})}{r^{\frac{d}{2(d+1)}}} \gtrsim_d 1.
    \]
    If, in addition, $K$ is centrally symmetric, then 
    \[\limsup_{r\to\infty} \frac{\disc\left(\AA_{rK}\right)}{r^{\frac{d}{2(d+1)}}\sqrt{\log r}} \lesssim_d 1;\hspace{2em} \liminf_{r\rightarrow\infty}\frac{\disc\left(\AA_{rK}\right)}{r^{\frac{d}{2(d+1)}}} \gtrsim_d 1.\]
\end{corollary}
\begin{proof}
    Let $\vol_d(\cdot)$ be the $d$-dimensional Lebesgue measure. Then, since $K-K$ is convex, and therefore, Jordan measurable, we have~\cite[Section 7.2]{Gruber-convgeom}
    \begin{equation}\label{eq:zeta-limit}
    \lim_{r\to \infty} \frac{\zeta_{K-K}(r)}{r^d} = \vol_d(K-K).
    \end{equation}
    By the the Brunn-Minkowski and Rogers-Shephard inequalities~\cite{RogersShephard57},  we have $2^d \le \vol_d(K-K) \le 4^d$. Therefore, for any $\vec{v} \in \R^d$, and any large enough $r$ we can assume that $|\Omega_{\vec{v}+rK}| \le \zeta_{K-K}(r) \asymp_d r^d$ (see Lemma~\ref{lem:zeta-maxshift} for the first inequality). Moreover, for a suitable constant $C_d$ depending only on $d$, and for all large enough $r$, choosing $s = (C_dr^d)^{\frac{1}{d+1}}$ gives us
    \[
    \zeta_{K-K}\left(\frac{r}{s}\right) \le \frac{C_d r^d}{s^d} = s.
    \]
    Therefore, $f(rK) \lesssim_d r^{\frac{d}{2(d+1)}}$. Similarly, for a suitable constant $c_d$ depending only on $d$, and all large enough $r$, choosing $s < (c_dr^d)^{\frac{1}{d+1}}$ gives us
    \[
    \zeta_{K-K}\left(\frac{r}{s}\right) \ge \frac{c_d r^d}{s^d} > s,
    \]
    implying that $f(rK) \gtrsim_d r^{\frac{d}{2(d+1)}}$, as well. So, for all large enough $r$, we have the estimate $f(rK) \asymp_d r^{\frac{d}{2(d+1)}}$, and the corollary follows from~\Cref{thm:generalbodiesthm}.
\end{proof}

More refined bounds on the growth rate of $\zeta_{K-K}(t)$ than the limit~\eqref{eq:zeta-limit} would automatically lead to sharper bounds on the discrepancy.

\section{Background}
Our proof will require several facts about factorization norms, which we will recount in~\Cref{sec:gamma2norm}, as well a variant of discrepancy which we define now.   

First, recall the definition of the discrepancy of $(\Omega, \mathcal{S})$ given in equation~\ref{eq:disc}
\begin{equation*}
    \disc\left(\mathcal{S}\right) = \min_{\chi: \Omega \to \{-1, 1\}}\max_{S \in \SS}\left|\sum_{\omega \in S}\chi(\omega)\right|.
\end{equation*}
Note that we can give an equivalent definition of $\disc\left(\mathcal{S}\right)$ with respect to the \emph{incidence} matrix $\mathbf{A}_{\mathcal{S}}$ of $(\Omega, \mathcal{S})$. Taking $\Omega = [n]$ and $\mathcal{S} = \{S_1, \ldots, S_m\}$ for simplicity, $\mathbf{A}_{\mathcal{S}}$ is the $m \times n$ zero-one matrix whose $(i,j)$ entry is equal to one if and only if $j \in S_i$ and is zero otherwise. It follows that $\disc(\mathcal{S})$ is equal to 
\begin{equation}\label{eq:disc-alt}
    \disc\left(\mathbf{A}_{\mathcal{S}}\right) = \min_{\mathbf{x} \in \{-1, +1\}^{n}}\left\|\mathbf{A}_{\mathcal{S}}\mathbf{x}\right\|_{\infty}.
\end{equation}

Next, let $\sigma: \Omega \to [|\Omega|]$ be a bijection. We think of $\sigma$ as giving an ordering to $\Omega$. The \emph{prefix discrepancy} of $\mathcal{S}$ with respect to $\sigma$ as is defined as
\[
\pdisc(\mathcal{S},\sigma) \eqdef
\min_{\chi:\Omega\to\{-1,+1\}}
\max_{j = 1}^{n}
\max_{S \in \mathcal{S}} \left|\sum_{\omega \in S: \sigma(\omega) \le j} \chi(\omega)\right|.
\]
Again $\pdisc(\mathcal{S},\sigma)$ has an equivalent definition $\pdisc\left(\mathbf{A}_{\mathcal{S}}\right)$ with respect to the incidence matrix $\mathbf{A}_{\mathcal{S}}$ of $\mathcal{S}$ with columns ordered according to $\sigma$: the first column indicates the incidences of $\sigma^{-1}(1)$, the second of $\sigma^{-1}(2)$, etc. Next, we can define the \emph{prefix discrepancy} of $\mathbf{A}_{\mathcal{S}}$ by 
\begin{equation}\label{eq:pdisc}
    \pdisc\left(\mathbf{A}_{\mathcal{S}}\right) \eqdef \min_{\mathbf{x} \in \{-1,+1\}^n} \max_{j = 1}^{n} \left\|\sum_{i=1}^j x_i \mathbf{a}_i\right\|_\infty,
\end{equation}
where $\mathbf{a}_i$ is the $i$-th column of $\mathbf{A}_{\mathcal{S}}$. 

\subsection{\texorpdfstring{$\gamma_2$}{gamma-2} Norm}\label{sec:gamma2norm}
The $\gamma_2$ factorization norm of a matrix $\mathbf{A}$, denoted  $\gamma_2(\mathbf{A})$, equals 
\[\gamma_2\left(\mathbf{A}\right) = \inf\{\|\mathbf{L}\|_{2\to\infty} \|\mathbf{R}\|_{1\to 2}: \mathbf{LR} = \mathbf{A}\}.\]
Here, $\|\mathbf{L}\|_{2\to\infty}$ is the $\ell_2 \to \ell_\infty$ operator norm, and equals the maximum $\ell_2$ norm of a row of $\vec{L}$, and $\|\vec{R}\|_{1\to 2}$ is the $\ell_1 \to \ell_2$ operator norm, and equals the maximum $\ell_2$ norm of a column of $\vec{R}$.

We extend this definition to set systems by writing $\gamma_2(\SS) = \gamma_2\left(\mathbf{A}_{\mathcal{S}}\right)$, where $\mathbf{A}_{\mathcal{S}}$ is, as above, the incidence matrix of $\SS$. Notice that $\gamma_2(\mathcal{S})$ is invariant under isomorphism and, similarly, $\gamma_2\left(\mathbf{A}_{\mathcal{S}}\right)$ is invariant under permuting the rows or columns of $\mathbf{A}_{\mathcal{S}}$.

We recall a few basic properties of the $\gamma_2$ norm. See Section~7~of~\cite{matouvsek2020factorization} for the proofs of these properties.

\begin{lemma}{\cite{matouvsek2020factorization}.}\label{lem:gamma2}
    The $\gamma_2$ norm satisfies the following properties.
    \begin{itemize} 
        \item (Triangle Inequality) For any matrices $\vec{A}$ and $\vec{B}$ of equal dimensions,
        \begin{equation}\label{eq:gamma2-triangle}
            \gamma_2\left(\mathbf{A} + \mathbf{B}\right) \le \gamma_2\left(\mathbf{A}\right) + \gamma_2\left(\mathbf{B}\right)
        \end{equation}
        \item (Union Property) For any two set systems $\SS$ and $\SS'$ on the same universe,
        \begin{equation}\label{eq:gamma2-union}
            \gamma_2(\SS \cup \SS')^2 \le \gamma_2(\SS)^2 + \gamma_2(\SS')^2.
        \end{equation}
        \item (Degree Bound) If $\mathcal{S}$ is a set system such that each element of the universe is contained in at most $t$ sets, then 
        \begin{equation}\label{eq:gamma2-degreebound}
            \gamma_2(\mathcal{S}) \le \sqrt{t}.
        \end{equation}
        \item (Size Bound) If all sets in a set system $\mathcal{S}$ have size at most $t$, then 
        \begin{equation}\label{eq:gamma2-sizebound}
            \gamma_2(\mathcal{S}) \le \sqrt{t}.
        \end{equation}
        \item (Disjoint Supports) If $\mathcal{S}$ and $\mathcal{S}'$ are set systems on disjoint universes, then 
        \begin{equation}
            \gamma_2\left(\mathcal{S} \cup \mathcal{S}'\right) = \max\left(\gamma_2\left(\mathcal{S}\right), \gamma_2\left(\mathcal{S}'\right)\right).
        \end{equation}
    \end{itemize}
\end{lemma}

We also state a property that is an immediate consequence of the triangle inequality when applied to set systems. A slightly less general variant was also stated as Property E in Section 7 of~\cite{matouvsek2020factorization}.
\begin{lemma}\label{lem:triangle-ss}
Suppose that $\SS$, $\SS_1$, $\SS_2$ are set systems on the same universe, so that each set $S$ in $\SS$ can be written as the disjoint union of sets $S_1 \in \SS_1$ and $S_2 \in \SS_2$. Then $\gamma_2(\SS) \le \gamma_2(\SS_1) + \gamma_2(\SS_2)$.

Similarly, suppose each set $S$ in $\SS$ can be written as $S_1 \setminus S_2$ where $S_1 \in \SS_1$, $S_2 \in \SS_2$, and $S_2 \subseteq S_1$. Then $\gamma_2(\SS) \le \gamma_2(\SS_1) + \gamma_2(\SS_2)$.
\end{lemma}
\begin{proof}
    Let $\vec{A},\vec{A}_1,\vec{A}_2$ be the incidence matrices of, respectively, $\SS, \SS_1,\SS_2$. Under the first assumption, we can re-arrange, and, if necessary, duplicate the rows of the matrices so that $\vec{A}=\vec{A}_1+\vec{A}_2$, and then the lemma follows from the triangle inequality property of $\gamma_2$, and the obvious property that duplicating and re-arranging rows does not change the $\gamma_2$ norm of a matrix. The proof under the second assumption is analogous, but we get the equation $\vec{A}=\vec{A}_1-\vec{A}_2$ instead.
\end{proof}

The relevance of the $\gamma_2$ norm to discrepancy theory stems from the following result, which is a constructive version of Theorem 3.4~from~\cite{matouvsek2020factorization}. 
\begin{theorem}[\cite{GSwalk}]\label{thm:gswalk}
    There exists a randomized polynomial time algorithm that, for any $m\times n$ matrix $\vec{A}$ with rational entries, computes in time polynomial in $m,n$, and the bit complexity of $\vec{A}$ a coloring $\vec{x} \in \{-1,+1\}^n$ such that, with high probability, 
    \[
    \|\vec{A}\vec{x}\|_\infty \lesssim \sqrt{\log 2m}\, \gamma_2(\vec{A}).
    \]
    In particular, there exists a randomized algorithm that, for any set system $(\Omega,\SS)$, computes in time polynomial in $|\Omega|$ and $|\SS|$, a coloring $\chi:\Omega\to\{-1,+1\}$ so that, with high probability,
    \[
    \disc(\SS,\chi) \lesssim \sqrt{\log 2|\SS|}\, \gamma_2(\vec{A}).
    \]
\end{theorem}
It is worth noting here that $\gamma_2(\vec{A})$ and $\gamma_2(\SS)$ also gives lower bounds on the hereditary discrepancy of $\vec{A}$ and $\SS$, respectively~\cite{matouvsek2020factorization}. We will not use this fact here.

We also need the following result of Banaszczyk.
\begin{theorem}{\cite[Lemma 2.5]{bana12}}\label{thm:bana-ss}
    Suppose $\mathbf{v}_1, \ldots, \mathbf{v}_n\in \R^m$ satisfy $\|\mathbf{v}_i\|_2 \le 1$ for all $i \in \{1, \ldots, n\}$.
    For any closed symmetric convex set $K\subseteq \R^m$ with Gaussian measure at least $1-\frac{1}{2n}$, there exists a coloring $\mathbf{x} \in \{-1,+1\}^n$ such that, for each $j \in \{1, \ldots, n\}$, $\sum_{i = 1}^j x_i \mathbf{v}_i \in 5K$.
\end{theorem}

The following theorem is an easy consequence of Theorem~\ref{thm:bana-ss}, which is implicit in the second author's work on the discrepancy of boxes and polytopes~\cite{nikolov-boxes}. We include the short proof here.

\begin{theorem}\label{thm:gamma2-pdisc}
    For any $m\times n$ matrix $\mathbf{A}$, 
    \[
    \pdisc\left(\mathbf{A}\right) \lesssim \sqrt{\log(2mn)} \,\gamma_2\left(\mathbf{A}\right).
    \]
    As a special case, for any set system $(\Omega, \mathcal{S})$, and any $\sigma:\Omega \to [|\Omega|]$
    \[
    \pdisc(\mathcal{S},\sigma) \lesssim \sqrt{\log(2mn)} \, \gamma_2(\mathcal{S}).
    \]    
\end{theorem}
\begin{proof}
    Suppose that the factorization $\mathbf{A} = \mathbf{L}\mathbf{R}$, where $\mathbf{L}$ is $m\times k$ and $\mathbf{R}$ is $k\times n$, achieves $\gamma_2(\mathbf{A})$. In particular, suppose that $\|\mathbf{L}\|_{2\to\infty} = \gamma_2(\mathbf{A})$ and $\|\mathbf{R}\|_{1\to 2} =1$. We can always make sure this is the case by multiplying $L$ by $\frac{\gamma_2(\mathbf{A})}{\|\mathbf{L}\|_{2\to\infty}}$ and $R$ by $\frac{\|\mathbf{L}\|_{2\to\infty}}{\gamma_2(\mathbf{A})} = \frac{1}{\|\mathbf{R}\|_{1\to 2}}$. 
    
    Now let $K$ be the convex body $K=\{\mathbf{y} \in \R^k: \|\mathbf{Ly}\|_\infty \le t\gamma_2(\mathbf{A})\}$, for a value of $t$ we will choose shortly. A standard calculation (see the proof of Theorem 3.4 of~\cite{matouvsek2020factorization}) shows that the Gaussian measure of $K$ is at least $(1-e^{-t^2/2})^m \ge 1-me^{-t^2/2}$, which is $1-\frac{1}{2n}$ for $t = \sqrt{2\ln(2mn)}$. We can then apply Theorem~\ref{thm:bana-ss} to $K$ and the column vectors $\mathbf{r}_1, \ldots, \mathbf{r}_n$ of $\mathbf{R}$, and conclude that there is a $\mathbf{x}\in \{-1,+1\}^n$ such that, for each $j \in \{1, \ldots, n\}$,
    \begin{align*}
    \sum_{i=1}^j x_i \mathbf{r}_i \in 5K
    &\iff
    \left\|\sum_{i=1}^j x_i \mathbf{L}\mathbf{r}_i\right\|_\infty \le 5\sqrt{2\ln(2mn)}\,\gamma_2(\mathbf{A})\\
    &\iff 
    \left\|\sum_{i=1}^j x_i \mathbf{a}_i\right\|_\infty \le 5\sqrt{2\ln(2mn)}\,\gamma_2(\mathbf{A}).
    \end{align*}
    Above, we first used the definition of $K$, and next we used the property $\mathbf{LR} = \mathbf{A}$. Since $j \in \{1, \ldots, n\}$ was arbitrary, this completes the proof.
    \end{proof}

We remark here that there is no known polynomial time algorithm to find the colorings guaranteed by Theorem~\ref{thm:gamma2-pdisc}.

\subsection{A Few More Facts}
Finally, our proof requires an additional definition, and a simple counting lemma.

Let us first introduce the concept of a \emph{maximal arithmetic progression}. We say that an arithmetic progression $\AP_d(\mathbf{a}, \mathbf{b},\ell)$ is \emph{maximal} inside $\Omega_{\mathbf{N}}$ for $\mathbf{N} = (N_1, \ldots, N_d)$, if $\mathbf{a} + \ell\cdot \mathbf{b}$ and $\mathbf{a} - \mathbf{b}$ are both not in $\Omega_{\mathbf{N}}$, i.e., the arithmetic progression is maximal if it cannot be extended in either direction while staying inside $\Omega_{\mathbf{N}}$. Notice that, for any $\mathbf{a}$ and $\mathbf{b}$ in $\Z^d$ there is at most one maximal $\AP_d(\mathbf{a}, \mathbf{b}, \ell)$. We use the notation $\MAP_{\mathbf{N}}(\mathbf{a},\mathbf{b})$ for the unique maximal arithmetic progression with step $\mathbf{b}$ containing $\mathbf{a}$. Note that we don't require $\mathbf{a}$ to be an endpoint of the $\MAP_{\mathbf{N}}(\mathbf{a},\mathbf{b})$. Equivalently, $\MAP_{\mathbf{N}}(\mathbf{a},\mathbf{b})$ is the intersection of the residue class of $\mathbf{a} \pmod{\mathbf{b}}$ with $\Omega_{\mathbf{N}}$. This alternative definition also makes clear that, for a fixed $\mathbf{b}$, the sets $\MAP_{\mathbf{N}}(\mathbf{a},\mathbf{b})$ partition $\Omega_{\mathbf{N}}$. 

The following is stated in the proof of Lemma 2.3 of Fox, Xu, Zhou~\cite{fox2024discrepancy}.
\begin{lemma}\label{lem:large-sets}
    For every $s \ge 2$, the quantity
    \[
    \left|\left\{\mathbf{b} \in \Z^d: \max_{\mathbf{a} \in \Z^d} |\MAP_{\mathbf{N}}(\mathbf{a},\mathbf{b})| \ge s\right\}\right|
    \] is bounded above by 
    \[
    \prod_{i=1}^d \left(\frac{4N_i}{s} + 1\right)
    \lesssim_d \max_{I \subseteq [d]}\prod_{i \in I}\left(\frac{N_i}{s}\right).
    \]
\end{lemma}

\section{Discrepancy Upper Bound}
We are now ready to prove the main theorem, restated below.
\main*
The structure of the proof is as follows. Let $\AA_{\mathbf{N}}$ be the set of arithmetic progressions and $\MM_{\mathbf{N}}$ be the corresponding set of maximal arithmetic progressions. We first show $\gamma_2\left(\MM_{\mathbf{N}}\right) \leq f(\mathbf{N})$ in~\Cref{lem:gamma2-map-ub} of Section~\ref{sec:gamma2-bound-map}. Roughly we do so by defining two set systems: $\MM_{> s}$, and $\MM_{\le s}$. $\MM_{> s}$ consists of all maximal APs $\MAP_{\mathbf{N}}(\mathbf{a},\mathbf{b})$ for which $|\MAP_{\mathbf{N}}(\mathbf{a}', \mathbf{b})| > s$ for at least one $\mathbf{a}'$ (not necessarily equal to $\mathbf{a}$). Lemma~\ref{lem:large-sets} allows us to bound the degree of  $\MM_{> s}$. Then $\MM_{\le s}$ consists of all $\MAP_{\mathbf{N}}(\mathbf{a},\mathbf{b})$ for which $|\MAP_{\mathbf{N}}(a',b)| \le s$ for all $a'$. Here every set has size at most $s$. Since $\MM_{\mathbf{N}} = \MM_{> s} \cup \MM_{\le s}$, using the $\gamma_2$ properties stated in Section~\ref{sec:gamma2norm} we can bound $\gamma_2\left(\MM_{\mathbf{N}}\right)$ by bounding $\gamma_2\left(\MM_{> s}\right)$ and $\gamma_2\left(\MM_{\le s}\right)$. 

This alone is sufficient for a proof of the discrepancy upper bound in~\Cref{thm:main}, given in Section~\ref{sec:nonconstructive}. We argue that there is an ordering $\sigma:\Omega_{\vec{N}} \to [|\Omega_{\vec{N}}|]$ such that $\disc(\AA_{\mathbf{N}}) \lesssim \pdisc(\MM_{\mathbf{N}},\sigma)$. This holds for the $\sigma$ given by the lexicographic order: $\sigma(x) < \sigma(y)$ if, for the smallest coordinate index $i$ for which $x_i \neq y_i$, we have $x_i < y_i$.

We also prove a tight upper bound $\gamma_2\left(\AA_{\mathbf{N}}\right)\lesssim f(\mathbf{N})$ on the $\gamma_2$ norm of arithmetic progressions in~\Cref{lem:gamma2-ap-ub} of Section~\ref{sec:gamma2-bound}. This gives another proof of the upper bound in~\Cref{thm:main}, and implies an efficient randomized algorithm to compute a coloring achieving the discrepancy upper bound via~\Cref{thm:gswalk}.

\subsection{Maximal Arithmetic Progressions}\label{sec:gamma2-bound-map}

As in the outline above, we first prove an upper bound on the set system $\MM_{\vec{N}}$ of all maximal APs.

\begin{lemma}\label{lem:gamma2-map-ub}
    Let $\MM_{\mathbf{N}} = \{\MAP_{\mathbf{N}}\left(\mathbf{a},\mathbf{b}\right): \mathbf{a},\mathbf{b}\in \Z^d\}$ be the set of all maximal APs in $\Omega_{\mathbf{N}}$. Then 
    \(\gamma_2(\MM_{\mathbf{N}}) \lesssim_d f\left(\mathbf{N}\right).\)
\end{lemma}
\begin{proof}
Let $s \ge 1$ be a parameter to be determined later.
Let $\MM_{\le s}$ consist of all $\MAP_{\mathbf{N}}(\mathbf{a},\mathbf{b})$ for which $|\MAP_{\mathbf{N}}(\mathbf{a'},\mathbf{b})| \le s$ for all $\mathbf{a}' \in \Z^d$. Since all sets in $\MM_{\le s}$ have cardinality at most $s$, the bound $\gamma_2(\MM_{\le s}) \le \sqrt{s}$ holds by the size bound of Lemma~\ref{lem:gamma2}.

Define the set $B$ as 
\[
B \eqdef \left\{\mathbf{b} \in \Z^d: \max_{\mathbf{a} \in \Z^d} |\MAP_{\mathbf{N}}(\mathbf{a},\mathbf{b})| > s\right\}.
\]
Let $\MM_{> s} = \MM_{\mathbf{N}}\setminus \MM_{\le s}$, and notice that \[\MM_{> s} \eqdef \{\MAP_{\mathbf{N}}(\mathbf{a},\mathbf{b}): \mathbf{a} \in \Z^d, \mathbf{b} \in B\}.\] Recall that, for any $\mathbf{b}$, the sets $\MAP_{\mathbf{N}}(\mathbf{a},\mathbf{b})$ partition $\Omega_{\mathbf{N}}$. Therefore, for any $\mathbf{b}$, $\MM_{\mathbf{N}}(\mathbf{b}) \eqdef \{\MAP_{\mathbf{N}}(\mathbf{a},\mathbf{b}): \mathbf{a} \in \Z^d\}$ is a set system of disjoint sets, i.e., it has maximum degree one. It follows that the maximum degree of $\MM_{> s} = \bigcup_{b \in B} \MM_{\mathbf{N}}(b)$ is at most $|B|$, and, by the degree bound of Lemma~\ref{lem:gamma2}, and using Lemma~\ref{lem:large-sets},  we get
\[
\gamma_2(\MM_{> s}) \le
\sqrt{|B|} \lesssim_d\left(\max_{I \subseteq [d]}\prod_{i \in I}{\frac{N_i}{s}}\right)^{\frac12}.
\]
Since $\MM_{\mathbf{N}} = \MM_{\le s} \cup \MM_{> s}$, the union property of Lemma~\ref{lem:gamma2} gives us
\begin{equation}\label{eq:gamma2-mm}
\gamma_2(\MM_{\mathbf{N}})^2 \lesssim_d s + \max_{I \subseteq [d]}\prod_{i \in I}\frac{N_i}{s}.
\end{equation}
To get the final bound, we need to pick the right value for $s$. 

First a definition: for $I \subseteq [d]$, let $P_I = \prod_{i \in I}N_i$. Choose $s \eqdef f(\mathbf{N})^2$ where $f\left(\textbf{N}\right) \coloneqq \max_{I \subseteq [d]} \left(P_I\right)^{\frac{1}{2|I| + 2}}$. Let $I^*$ such that $f\left(\textbf{N}\right) = \left(P_{I^*}\right)^{\frac{1}{2|I^*| + 2}}$. Then,
\begin{align*}
    s + \max_{I \subset [d]}\frac{P_I}{s^{|I|}} &= \left(P_{I^*}\right)^{\frac{1}{|I^*| + 1}} + \max_{I \subseteq [d]}\frac{P_I}{P_{I^*}^{\frac{|I|}{|I^*| + 1}}}\\
    &\leq \left(P_{I^*}\right)^{\frac{1}{|I^*| + 1}} + \max_{I \subseteq [d]}\frac{P_I}{P_{I}^{\frac{|I|}{|I| + 1}}}\\
    &= \left(P_{I^*}\right)^{\frac{1}{|I^*| + 1}} + \max_{I \subseteq [d]}\left(P_I\right)^{\frac{1}{|I| + 1}} = 2f\left(\textbf{N}\right)^2
\end{align*}
where the inequality follows from the definition of $f\left(\textbf{N}\right)$. 
\end{proof}

\subsection{Upper Bound via Prefix Discrepancy}\label{sec:nonconstructive}

Here we prove the discrepancy upper bound part of~\Cref{thm:main}, with an argument that does not imply the constructive part. Formally, we show the following theorem.

\begin{theorem}\label{thm:main-ub}
    For any positive integer $d$ and vector $\mathbf{N} = \left(N_1, \ldots, N_d\right)$, where each $N_i$ is an integer greater than $1$, let $\AA_{\mathbf{N}}$ be the family of all arithmetic progression defined on the universe $\Omega_{\mathbf{N}} $. The discrepancy of $\AA_{\mathbf{N}}$ is bounded above as 
    \begin{equation*}
        \disc\left(\AA_{\mathbf{N}}\right) \lesssim_d
        \sqrt{\log |\Omega_{\mathbf{N}}|} \cdot f\left(\mathbf{N}\right).
    \end{equation*}
\end{theorem}

While we don't need \Cref{thm:main-ub} to prove \Cref{thm:main}, we include its proof here because of its simplicity, and because we believe the technique of using prefix discrepancy to prove bounds on the discrepancy of set systems is useful and may have further application. It is also the technique we use in proving the discrepancy upper bounds in Theorem~\ref{thm:generalbodiesthm}, and was previously used by Nikolov in his upper bound on the discrepancy of axis-aligned boxes~\cite{nikolov-boxes}.

Before we prove~\Cref{thm:main-ub}, we first establish a useful lemma. For the statement of the lemma, let $\preclex$ be the lexicographic order on $\Z^d$, i.e., $\vec{x} \preclex \vec{y}$ if $\vec{x} = \vec{y}$, or if $x_i < y_i$ holds for the smallest integer $i$ for which $x_i \neq y_i$. 

\begin{lemma}\label{lem:lex}
    For any arithmetic progression $\AP_d(\mathbf{a},\mathbf{b},\ell)$ in $\Z^d$, there exist $\vec{x}, \vec{y} \in \AP_d(\mathbf{a},\mathbf{b},\ell)$ so that
    \[
    \AP_d(\mathbf{a},\mathbf{b},\ell) = 
    \{\vec{a} + i \vec{b}: i \in \Z\} \cap \{\vec{z}: \vec{x} \preclex \vec{z}\preclex \vec{y}\}.
    \]
\end{lemma}
\begin{proof}
    Fix $\vec{a}$ and $\vec{b} \neq 0$ (the lemma is trivial for $\vec{b} = 0$), and let $k$ be the smallest integer such that $b_k \neq 0$. We will assume that $b_k > 0$. This is without loss of generality, since for $\vec{a}' \eqdef \mathbf{a} + (\ell-1)\vec{b}$ we have 
    \(
    \AP_d(\mathbf{a},\mathbf{b},\ell) = \AP_d(\mathbf{a}',-\mathbf{b},\ell),
    \)
    and 
    \(
    \{\vec{a} + i \vec{b}: i \in \Z\} = \{\vec{a} - i \vec{b}: i \in \Z\},
    \)
    and, therefore, we can replace $\vec{a}$ with $\mathbf{a}'$ and $\vec{b}$ with $-\vec{b}$.

    Let $\vec{a}' = \mathbf{a} + (\ell-1)\vec{b}$ as above. Note that 
    \[
    \AP_d(\mathbf{a},\mathbf{b},\ell)
    = 
    \{\vec{a} + i \vec{b}: i \in \Z, i \ge 0\} 
    \cap 
    \{\vec{a}' + i \vec{b}: i \in \Z, i \le 0\}.
    \]
    We claim that 
    \begin{equation}\label{eq:AP-repr-pref}
    \{\vec{a} + i \vec{b}: i \in \Z, i \ge 0\}  = 
    \{\vec{a} + i \vec{b}: i \in \Z\} \cap\{\vec{z}: \vec{a} \preclex \vec{z}\},
    \end{equation}
    and 
    \begin{equation}\label{eq:AP-repr-suf}
    \{\vec{a}' + i \vec{b}: i \in \Z, i \le 0\}  = 
    \{\vec{a} + i \vec{b}: i \in \Z\} \cap\{\vec{z}: \vec{z}\preclex \vec{a}'\},
    \end{equation}
    and then the lemma follows from these two equation after setting $\vec{x} \eqdef \vec{a}$ and $\vec{y} \eqdef \vec{a}'$.
    
    Let us first prove \eqref{eq:AP-repr-pref}. 
    Since $b_j = 0$ for $j < k$ and $b_k > 0$, we have that, for any $\vec{z} = \vec{a} + i \vec{b}$, $i \ge 0$ if and only if $a_k \le z_k$. Moreover any such $\vec{z}$ satisfies $z_j = a_j$ for all $j < k$. This means that $\vec{a} \preclex \vec{z}$ if and only if $i \ge 0$. 

    The proof of \eqref{eq:AP-repr-suf} is analogous. First note that $\{\vec{a} + i \vec{b}: i \in \Z\}=\{\vec{a}' + i \vec{b}: i \in \Z\}$. Then, for any $\vec{z} = \vec{a}'+ i \vec{b}$, $i \le 0$ if and only if $a'_k \ge z_k$. Moreover any such $\vec{z}$ satisfies $z_j = a'_j$ for all $j < k$. This means that $\vec{z} \preclex \vec{a}'$ if and only if $i \ge 0$. 
\end{proof}

In addition to the proof of~\Cref{thm:main-ub} below, we also use Lemma~\ref{lem:lex} in the proof of the upper bound in Theorem~\ref{thm:generalbodiesthm}.

\begin{proof}[Proof of~\Cref{thm:main-ub}]
    We claim that there exists a bijection $\sigma:\Omega_{\mathbf{N}} \to [|\Omega_{\mathbf{N}}|]$ for which $\disc(\AA_{\mathbf{N}}) \le 2 \pdisc(\MM_{\mathbf{N}}, \sigma)$. Once we establish this, \Cref{thm:main} follows from~\Cref{thm:gamma2-pdisc} and \Cref{lem:gamma2-map-ub}.

    Define $\sigma(\mathbf{x})$ to be the rank of $\mathbf{x}$ in the lexicographic ordering of $\Omega_{\mathbf{N}}$: $\sigma(\mathbf{x}) = i$ if and only if there are exactly $i-1$ elements $\vec{y}$ of $\Omega_{\mathbf{N}}$ such that $\vec{y} \preclex\vec{x}$ and $\vec{y} \neq \vec{x}$.

    Consider an arithmetic progression $\AP_d(\mathbf{a},\mathbf{b},\ell)$ in $\AA_{\vec{N}}$. It follows from Lemma~\ref{lem:lex} that there exist integers $i$ and $j$ such that 
    \begin{equation}\label{eq:AP-repr}
        \AP_d(\mathbf{a},\mathbf{b},\ell) = \MAP_{\mathbf{N}}(\mathbf{a},\mathbf{b}) \cap \{\mathbf{x}\in \Omega_{\mathbf{N}}: i \le \sigma(\mathbf{x}) \le j\}.
    \end{equation}
    Indeed, if $\vec{x}$ and $\vec{y}$ are the points guaranteed by the Lemma~\ref{lem:lex}, we can take $i = \sigma(\vec{x})$ and $j = \sigma(\vec{y})$. Then equation~\eqref{eq:AP-repr} follows from the lemma and the observation that 
    \(
    \{\vec{a} + i\vec{b}: i \in \Z\} \cap \Omega_{\vec{N}}
    = \MAP_{\mathbf{N}}(\mathbf{a},\mathbf{b}).
    \)
    
    Now let $\chi:\Omega_{\mathbf{N}} \to \{-1,+1\}$ achieve $\pdisc(\MM_{\mathbf{N}},\sigma)$. By equation~\eqref{eq:AP-repr}, for any arithmetic progression $\AP_d(\mathbf{a},\mathbf{b},\ell)$ in $\Omega_{\mathbf{N}}$ there exist $i$ and $j$ such that
    \begin{align*}
    \left|\sum_{\mathbf{z} \in \AP_d(\mathbf{a},\mathbf{b},\ell)} \chi(\mathbf{z}) \right|
    &=
    \left|\sum_{\substack{\mathbf{z} \in \MAP_{\mathbf{N}}(\mathbf{a},\mathbf{b})\\ i \le \sigma(\mathbf{z}) \le j}} \chi(\mathbf{z}) \right|\\
    &= \left|\sum_{\substack{\mathbf{z} \in \MAP_{\mathbf{N}}(\mathbf{a},\mathbf{b})\\ \sigma(\mathbf{z}) \le j}} \chi(\mathbf{z})  - \sum_{\substack{\mathbf{z} \in \MAP_{\mathbf{N}}(\mathbf{a},\mathbf{b})\\ \sigma(\mathbf{z}) < i}} \chi(\mathbf{z})\right|\\
    &\le \left|\sum_{\substack{\mathbf{z} \in \MAP_{\mathbf{N}}(\mathbf{a},\mathbf{b})\\ \sigma(\mathbf{z}) \le j}} \chi(\mathbf{z})\right|  + \left|\sum_{\substack{\mathbf{z} \in \MAP_{\mathbf{N}}(\mathbf{a},\mathbf{b})\\ \sigma(\mathbf{z}) < i}} \chi(\mathbf{z})\right|\\
    &\le 2\pdisc(\MM_{\mathbf{N}},\sigma).
    \end{align*}
    Since $\AP_d(\mathbf{a},\mathbf{b},\ell)$ was arbitrary, this proves that $\disc(\AA_{\mathbf{N}}) \le 2\pdisc(\MM_{\mathbf{N}},\sigma)$, and the theorem follows from Theorem~\ref{thm:gamma2-pdisc} and Lemma~\ref{lem:gamma2-map-ub}.
\end{proof}

\subsection{Bound on \texorpdfstring{$\gamma_2\left(\AA_{\mathbf{N}}\right)$}{gamma-two of AP}}\label{sec:gamma2-bound}

In this section, we give a tight bound on $\gamma_2(\AA_{\mathbf{N}})$. This bound and Theorem~\ref{thm:gswalk} then imply Theorem~\ref{thm:main}, and, in particular, imply an efficient algorithm to find the low discrepancy coloring guaranteed to exist by Theorem~\ref{thm:main-ub}. 

The lemma below and its proof are a generalization of Proposition~7.1 of Matou\v{s}ek, Nikolov, and Talwar~\cite{matouvsek2020factorization}.

\begin{lemma}\label{lem:gamma2-ap-ub}
    For $N_1, ..., N_d \in \Z$, let $\mathbf{N} = \left(N_1, ..., N_d\right)$. Let $\AA_{\mathbf{N}}$ be the set of all arithmetic progressions in $\Omega_{\mathbf{N}}$. Then 
    \(
    \gamma_2(\AA_{\mathbf{N}}) \lesssim_d f\left(\mathbf{N}\right).
    \)
\end{lemma}
\begin{proof}
    Without loss of generality, we will assume $N_1 \ge N_2 \ldots \ge N_d$. Moreover, we will assume that each $N_i$ is a power of $2$. This is also without loss of generality, since rounding up each $N_i$ to the nearest power of $2$ changes $f(\vec{N})$ by at most a constant factor.
    
    Let $\mathcal{PA}_{\vec{N}}$ be the set of prefix-maximal arithmetic progressions in $\Omega_{\vec{N}}$, i.e., arithmetic progressions $\AP(\vec{a},\vec{b},\ell)$ for which $\vec{a}-\vec{b} \not \in \Omega_{\vec{N}}$. Clearly any arithmetic progression in $\AA_{\vec{N}}$ is the set difference of at most two prefix-maximal arithmetic progressions in $\mathcal{PA}_{\vec{N}}$, so $\gamma_2\left(\AA_{\vec{N}}\right) \le 2 \gamma_2(\mathcal{PA}_{\vec{N}})$ by Lemma~\ref{lem:triangle-ss}. Therefore, it suffices to prove that $\gamma_2(\mathcal{PA}_{\vec{N}}) \lesssim_d f(\vec{N})$. 

    We prove this bound by induction on $|\Omega_{\vec{N}}|$, with the inductive hypothesis 
    \(
    \gamma_2(\mathcal{PA}_{\vec{N}}) \le C_d f\left(\mathbf{N}\right),
    \)
    where $C_d$ is a constant depending on $d$ that we will decide later. 
    The base case $|\Omega_{\vec{N}}| = 1$ is trivial, and the rest of the proof establishes the inductive step.
    
    Let $N'_1 \eqdef N_1/2$ (recall that we assumed $N_1$ is a power of $2$). Let also $\vec{N}' \eqdef (N'_1, N_2, \ldots, N_d)$, and define $\Omega' \eqdef \Omega_{\vec{N}'} = [N'_1]\times [N_2]\times \ldots \times [N_d]$, and $\Omega'' \eqdef \vec{v} + \Omega_{\vec{N}'}$, where $\vec{v}$ is the vector with $N_1'$ in the first coordinate, and $0$ in every other coordinate. Notice that $\Omega'$ and $\Omega''$ partition $\Omega_{\vec{N}}$. 
    Let $\PA'$ and $\PA''$ be, respectively, the set of all prefix-maximal arithmetic progressions in $\Omega'$, and the set of all prefix-maximal arithmetic progressions in $\Omega''$. Note that $\PA'$ and $\PA''$ are isomorphic: for every $\AP_d(\vec{a},\vec{b},\ell) \in \PA''$, $\AP_d(\vec{a}-\vec{v},\vec{b},\ell) \in \PA'$, and, conversely, for any $\AP_d(\vec{a},\vec{b},\ell) \in \PA'$, $\AP_d(\vec{a} + \vec{v},\vec{b},\ell) \in \PA''$. Moreover, $\PA' = \PA_{\vec{N}'}$ by definition. Similarly, let $\MM'$ and $\MM''$ be, respectively, the set of maximal arithmetic progressions in $\Omega'$ and in $\Omega''$. Again, $\MM'$ and $\MM''$ are both isomorphic to $\mathcal{M}_{\vec{N}'}$, the set of all maximal arithmetic progressions in $\Omega_{\vec{N}'}$.

    Take $\AP_d(\vec{a},\vec{b},\ell)$ to be a prefix-maximal arithmetic progression in $\PA_{\mathbf{N}}$ that is not fully contained in $\Omega'$ or $\Omega''$. Notice that this means $b_1 \neq 0$. If $b_1 > 0$, this is equivalent to $a_1 \le N'_1$ and $a_1 + (\ell-1)b_1 > N'_1$. Then $\AP_d(\vec{a},\vec{b},\ell) \cap \Omega' \in \MM'$, and $\AP_d(\vec{a},\vec{b},\ell) \cap \Omega'' \in \PA''$. Similarly, if $b_1 < 0$, then $\AP_d(\vec{a},\vec{b},\ell)$ not being fully contained in  $\Omega'$ or $\Omega''$ is equivalent to $a_1 > N'_1$ and $a_1 + (\ell-1)b_1 \le N'_1$. Then $\AP_d(\vec{a},\vec{b},\ell) \cap \Omega'' \in \MM''$, and $\AP_d(\vec{a},\vec{b},\ell) \cap \Omega' \in \PA'$. 
    
    In summary, one of the following cases holds for any $\AP_d(\vec{a},\vec{b},\ell) \in \PA_{\vec{N}}$:
    \begin{enumerate}
        \item $\AP_d(\vec{a},\vec{b},\ell) \in \PA'$ or $\AP_d(\vec{a},\vec{b},\ell) \in \PA''$;
        
        \item $\AP_d(\vec{a},\vec{b},\ell) = M' \cup A''$, where $M' \in \MM'$ and $A'' \in \PA''$;
        
        \item $\AP_d(\vec{a},\vec{b},\ell) = M'' \cup A'$, where $M'' \in \MM''$ and $A' \in \PA'$.
    \end{enumerate}
    Let $\SS_1$ be the set system on $\Omega_{\vec{N}}$ with sets $\PA'\cup \PA''$. By the disjoint supports property in Lemma~\ref{lem:gamma2}, and the fact that $\PA'$ and $\PA''$ are both isomorphic to $\PA_{\vec{N}'}$, we have $\gamma_2(\SS_1) = \gamma_2(\PA_{\vec{N}'})$. Similarly, let $\SS_2$ be the set system on $\Omega_{\vec{N}}$ with sets $\MM'\cup \MM''$. Again by the disjoint supports property in Lemma~\ref{lem:gamma2}, $\gamma_2(\SS_1) = \gamma_2(\MM_{\vec{N}'})$. By the observations above, each $\AP_d(\vec{a},\vec{b},\ell) \in \PA_{\vec{N}}$ can be written as $S_1 \cup S_2$ for two disjoint sets $S_1 \in \SS_1$ and $S_2 \in \SS_2$ (where we assume, without loss of generality, that the empty set is contained in $\SS_2$). By Lemma~\ref{lem:triangle-ss}, this gives us
    \begin{equation}\label{eq:gamma2-recurse}
        \gamma_2(\PA_{\vec{N}}) \le \gamma_2(\mathcal{M}_{\vec{N}'}) + \gamma_2(\PA_{\vec{N}'}) 
    \le (C'_d + C_d) f(\vec{N}').
    \end{equation}
    For the second inequality, we used the inductive hypothesis and Lemma~\ref{lem:gamma2-map-ub}, with $C'_d$ defined as the implied constant in the upper bound on $\gamma_2(\mathcal{M}_{\vec{N}})$ given by Lemma~\ref{lem:gamma2-map-ub}. To complete the proof, we claim that
    \begin{equation}\label{eq:f-decr}
        f(\vec{N}')\le 2^{-\frac{1}{(2d+2)(2d+4)}}f(\vec{N}).
    \end{equation}
    To see this, suppose that $f(\vec{N}') = \left(\prod_{i \in I}N'_i\right)^{\frac{1}{2|I| + 2}}$ for some $I \subseteq [d]$. If $1 \in I$, then, since $N'_1 = N_1/2$, it is clear that 
    \[
    f(\vec{N}') = \left(\prod_{i \in I}N'_i\right)^{\frac{1}{2|I| + 2}} 
    = 
    2^{-\frac{1}{2|I| + 2}}\left(\prod_{i \in I}N_i\right)^{\frac{1}{2|I| + 2}} \le 
    2^{-\frac{1}{2|I| + 2}} f(\vec{N}).
    \]
    Since $|I| \le d$, this is more than enough to prove \eqref{eq:f-decr}.
    If $1 \not \in I$, then notice that
    \begin{align*}
       f(\vec{N}')^2 = \left(\prod_{i \in I}N'_i\right)^{\frac{1}{|I| + 1}} 
       &= 
    \left(\frac{\left(\prod_{i \in I}N_i\right)^{\frac{1}{|I| + 1}}}{N_1}\right)^{\frac{1}{|I| + 2}}\left(\prod_{i \in I\cup \{1\}}N_i\right)^{\frac{1}{|I| + 2}}\\
    &\le
    N_1^{-\frac{1}{(|I|+1)(|I|+2)}} f(\vec{N}),
    \end{align*}
    where in the first equation we used that $N_i' = N_i$ for all $i \neq 1$, and in the inequality we used that $N_1 \ge N_i$ for all $i$. Now \eqref{eq:f-decr} follows from $N_1 \ge 2$ and $|I| \le d$.
    
    Using \eqref{eq:gamma2-recurse} and \eqref{eq:f-decr}, we see  the bound $\gamma_2(\PA_{\vec{N}}) \le C_d f(\vec{N})$ holds as long as $C_d \ge C'_d / (2^{\frac{1}{(2d+2)(2d+4)}}-1)$. This completes the proof of the inductive step.
\end{proof}

Theorem~\ref{thm:main} is now a direct consequence of Lemma~\ref{lem:gamma2-ap-ub} and Theorem~\ref{thm:gswalk}.

We remark here that Lemma~\ref{lem:gamma2-ap-ub} is tight up to constants depending on the dimension. This follows from the proof of Theorem 5.1 of~\cite{fox2024discrepancy}. An inspection of the proof shows that it gives a lower bound on the average squared discrepancy of some collection of APs in $\Omega_{\vec{N}}$, and that the lower bounds holds for any coloring $\chi:\Omega_{\vec{N}} \to \R$ with $\sum_{\vec{z} \in \Omega_{\vec{N}}} |\chi(\vec{z})|^2 = |\Omega_{\vec{N}}|$. Thus, the lower bound in Theorem 5.1 of~\cite{fox2024discrepancy} in fact gives a lower bound on the smallest singular value of a matrix $\vec{A}$ whose rows are indicator vectors of APs in $\Omega_{\vec{N}}$. In particular, if the number of rows of $\vec{A}$ is $M$, then the lower  bound is 
\[
\sigma_{\min}(\vec{A}) \gtrsim_d \frac{f(\vec{N})\sqrt{M}}{\sqrt{ |\Omega_{\vec{N}}| }}.
\]
The nuclear norm (sum of singular values) of $\vec{A}$, denoted $\|\vec{A}\|_*$, is then bounded as $\|\vec{A}\|_* \gtrsim_d f(\vec{N})\sqrt{M|\Omega_{\vec{N}}|}$, which implies $\gamma_2(\AA_{\vec{N}}) \ge \gamma_2(\vec{A}) \gtrsim_d f(\vec{N})$ by equation~(4)~of~\cite{matouvsek2020factorization}.

\section{Arithmetic Progressions in a Convex Set}

Let us now consider the general setting of convex bodies $K$ in $\R^d$. Let $\Omega_K$ be the integer points in $K$, i.e., $\Omega_K \eqdef \Z^d \cap K$. Let $\AA_K$ be the set system of all arithmetic progressions restricted to $\Omega_K$. We want to prove the following.

\generalbodiesthm*

\subsection{Upper Bound}

In analogy with the $d$-dimensional axis-aligned box, let $\MAP_K(\mathbf{a},\mathbf{b})$ be an arithmetic progression inside $K$ containing $\mathbf{a}$ and with difference $\mathbf{b}$ that cannot be extended further in either direction, i.e., a maximal arithmetic progression in $K$. We have the following bound.

\begin{lemma}\label{lem:generalK-MAP}
    Let $\MM_K := \{\MAP_K(\mathbf{a},\mathbf{b}): \mathbf{a},\mathbf{b} \in \Z^d\}$ be the set of all maximal APs in $\Omega_K$. Then $\gamma_2(\MM_K) \lesssim f(K)$.
\end{lemma}
\begin{proof}
    The proof follows along the same lines as that of Lemma~\ref{lem:gamma2-map-ub}. Fix some $s \ge 1$ to be determined later, and let $\MM_{\le s}$ consist of all $\MAP_K(\mathbf{a},\mathbf{b})$ for which $|\MAP_K(\mathbf{a}',\mathbf{b})| \le s$ for all $\mathbf{a}' \in \Z^d$. As in Lemma~\ref{lem:gamma2-map-ub}, $\gamma_2(\MM_{\le s}) \le \sqrt{s}$ because all sets in $\MM_{\le s}$ have cardinality at most $s$.

    Then let $\MM_{> s} := \MM_K \setminus \MM_{\le s}$. Defining 
    \[
     B \eqdef \left\{\mathbf{b} \in \Z^d: \max_{\mathbf{a} \in \Z^d} |\MAP_{K}(\mathbf{a},\mathbf{b})| > s\right\},
    \]
    we have $\MM_{> s} = \{\MAP_K(\mathbf{a},\mathbf{b}): \mathbf{a} \in \Z^d, \mathbf{b} \in B\}$. Analogous to Lemma~\ref{lem:gamma2-map-ub}, we have $\gamma_2(\MM_{> s}) \le \sqrt{|B|}$. 

    We estimate $|B|$. Take some $\mathbf{b} \in B$ and $\mathbf{a} \in \Z^d$ such that $|\MAP_{K}(\mathbf{a},\mathbf{b})| > s$. Let $\mathbf{x}$ and $\mathbf{y}$ be the endpoints of $\MAP_{K}(\mathbf{a},\mathbf{b})$, i.e., $\mathbf{x},\mathbf{y} \in \MAP_{K}(\mathbf{a},\mathbf{b})$, but $\mathbf{x}-\mathbf{b} \not \in \MAP_{K}(\mathbf{a},\mathbf{b})$, and $\mathbf{y}+\mathbf{b} \not \in \MAP_{K}(\mathbf{a},\mathbf{b})$. Then $\mathbf{y} - \mathbf{x} = \ell \mathbf{b}$ for some $\ell \ge s$. Since both $\mathbf{x}$ and $\mathbf{y}$ are in $K$, this means that $s\mathbf{b} \in K-K$, or, equivalently, $\mathbf{b} \in \frac{1}{s} (K-K)$. We then have that $|B| \le \zeta_{K-K}\left(\frac{1}{s}\right)$.

    Combining the bounds above, we get
    \[
    \gamma_2(\MM_K)^2 =\gamma_2(\MM_{\le s}\cup \MM_{> s})^2
    \le s + \zeta_{K-K}\left(\frac{1}{s}\right).
    \]
    In particular, we have 
    \[
    \gamma_2(\MM_K)^2
    \le 2\inf\left\{s: s \ge \zeta_{K-K}\left(\frac{1}{s}\right)\right\}
    = 2 f(K)^2,
    \]
    as we needed to prove.
\end{proof}

Next we prove the upper bound in Theorem~\ref{thm:generalbodiesthm}, restated as the next theorem.

\begin{theorem}\label{thm:generalbodies-ub}
    For any $d$-dimensional convex body $K$, any $\vec{v} \in \R^d$, $\Omega_{\vec{v} + K} \eqdef \Z^d \cap (\vec{v} + K)$, and the set system $\AA_{\vec{v} + K}$ of all arithmetic progressions on the universe $\Omega_{\vec{v} + K}$, we have
    \[
    \disc(\AA_{\vec{v} + K}) \lesssim \sqrt{\log (1+|\Omega_{\vec{v} + K}|)}\cdot f(K).
    \]
\end{theorem}
\begin{proof}
    Let $K' = \vec{v} + K$ and note that $K'-K' = K-K$, so $f(K') = f(K)$. The proof is analogous to the proof of Theorem~\ref{thm:main} via prefix discrepancy. We construct a bijection $\sigma:\Omega_{K'} \to [|\Omega_{K'}|]$ such that $\disc(A_{K'}) \le 2 \pdisc(\MM_{K'},\sigma)$. The theorem then follows from Theorem~\ref{thm:gamma2-pdisc} and Lemma~\ref{lem:generalK-MAP}. 

    As in the proof of Theorem~\ref{thm:main}, we order $\Omega_{K'}$ lexicographically (i.e., according to $\preclex$), and let $\sigma(\vec{x})$ be the rank of $\vec{x} \in \Omega_{K'}$ in this ordering. Then it follows from Lemma~\ref{lem:lex} that for any arithmethic progression $\AP_d(\vec{a}, \vec{b}, \ell)$ in $\AA_{K'}$ there exist integers $i$ and $j$ for which 
    \[
    \AP_d(\vec{a}, \vec{b}, \ell) = \MAP_K(\vec{a},\vec{b}) \cap \{\vec{x}: i \le \sigma(\vec{x}) \le j\}.
    \]
    Then $\disc(A_{K'}) \le 2 \pdisc(\MM_{K'},\sigma)$ follows from an analogous calculation as the one in the proof of Theorem~\ref{thm:main}. 
\end{proof}

\subsection{Lower Bound}
We would like a lower bound for $\disc(\AA_K)$ with respect to $f(K)$ in~\Cref{thm:generalbodiesthm}. However, the inequality is not true when $K$ is a ``narrow" convex body which avoids all the integer points. For example, we can take $K = [0.25,0.75] \times [0,1]$ in two dimensions, which contains no integer points so $\disc(\AA_K) = 0$, while $f(K) > 0$. Instead, we present a lower bound for $\sup_{\mathbf{v} \in \R^d} \disc(\AA_{\mathbf{v} + K})$ with respect to $f(K)$ using a Fourier analytic proof. The result and its proof generalizes Roth's 1/4-theorem for one-dimensional arithmetic progressions and the lower bound of Fox, Xu, Zhou~\cite{fox2024discrepancy}. 

First, we need a couple of geometric lemmas. The first relates the number of integer points in $\vec{v}+K$ to the number of integer points in ${K-K}$. It can be seen as a counting version of the Rogers-Shephard inequality. We note that other counting versions of the inequality are known~\cite{discrete-rogsh}.
\begin{lemma}\label{lem:zeta-maxshift}
    For any convex body $K$ in $\R^d$,  we have the inequalities
    \[
    \max_{\vec{v} \in \R^d} |\Z^d \cap (\vec{v} + K)| 
    \le |\Z^d \cap (K-K)| 
    \lesssim_d \max_{\vec{v} \in \R^d} |\Z^d \cap (\vec{v} + K)|.
    \]
\end{lemma}
\begin{proof}
    Let us take any $\vec{v} \in \R^d$. To prove the first inequality, we need to show that 
    \(
    |\Z^d \cap (\vec{v} + K)| \le |\Z^d \cap (K-K)|.
    \)
    This is trivial if $\Z^d \cap (\vec{v} + K) = \emptyset$, so let us assume otherwise. Fix any $\vec{x} \in \Z^d \cap (\vec{v} + K)$. Then 
    \[
    |\Z^d \cap (\vec{v} + K)| = |\{\vec{y}-\vec{x}: \vec{y} \in \Z^d \cap (\vec{v} + K)\}| \le |\Z^d \cap (K-K)|,
    \]
    where the last inequality follows since $\vec{y}-\vec{x} \in K - K$ for any $\vec{x},\vec{y} \in \vec{v}+K$. This proves the first inequality.

    Let us now prove the second inequality. Without loss of generality, we can assume that $\vec{0} \in K$, since both sides of the inequality are invariant to shifting $K$. By the standard volumetric argument, there exist vectors $\vec{v}_1, \ldots, \vec{v}_N$ such that $K-K \subseteq \bigcup_{i =1}^N (\vec{v}_i + K)$, where 
    \[
    N \le \frac{2^d\vol_d\left((K-K) + \frac{1}{2}K \right)}{\vol_d(K)}.    
    \]
    Here $\vol_d(\cdot)$ is the $d$-dimensional Lebesgue measure. For a proof, see \cite[Chapter 4]{convgeom-book}. Using the assumption that $\vec{0} \in K$, we have that $K\subseteq K-K$, so $(K-K) + \frac{1}{2}K\subseteq \frac32(K-K)$, and we have
    \[
    N\le \frac{3^d\vol_d(K-K)}{\vol_d(K)} \le 12^d,
    \]
    where the final bound follows from the Rogers-Shephard inequality~\cite{RogersShephard57}. We now have that, for at least one of the $\vec{v}_i$ vectors, by the pigeonhole principle
    \[
    |\Z^d \cap (\vec{v}_i + K)| \ge \frac{1}{N} |\Z^d \cap (K-K)|
    \gtrsim_d |\Z^d \cap (K-K)|,
    \]
    proving the second inequality.
\end{proof}

We also need the following lemma, which is Proposition~4 of Gillet and Soul\'{e}~\cite{gillet1991number}. We include a short proof, which also gives better constants.
\begin{lemma}\label{lem:zeta-scaling}
    For any centrally symmetric convex body $K$ in $\R^d$, i.e., one such that $K = -K$, and any $t \ge 1$, we have
    \[
     \zeta_K(1) \leq \zeta_K(t) \lesssim_d t^d\zeta_K(1).
    \]
\end{lemma}
\begin{proof}
    The first inequality is trivial for $t \ge 1$ so we prove the second one. The standard volumetric argument (referring, again, to~\cite[Chapter 4]{convgeom-book} for a proof) gives that there exist vectors $\vec{v}_1, \ldots, \vec{v}_N$ such that $tK \subseteq \bigcup_{i =1}^N (\vec{v}_i + \frac12 K)$, where 
    \[
    N \le \frac{2^d\vol_d\left(tK + \frac{1}{4}K \right)}{\vol_d(\frac12 K)} =  (4t+1)^d.
    \]
    Therefore, at least one of the $\vec{v}_i$ vectors, by the pigeonhole principle, 
    \[
    \left|\Z^d \cap (\vec{v}_i + \frac12 K)\right| \ge \frac{1}{N} |\Z^d \cap tK|
    \ge \frac{1}{(4t+1)^d}\zeta_K(t).
    \]
    The first inequality in Lemma~\ref{lem:zeta-maxshift} then give us that $|\Z^d \cap (\vec{v}_i + \frac12 K)| \le \zeta_K(1)$, since $\frac12K - \frac12 K = K$ by the symmetry of $K$. We thus get 
    \[
    \zeta_K(t) \le (4t+1)^d \zeta_K(1) \le 5^d t^d \zeta_K(1),
    \]
    where we used $t \ge 1$.
\end{proof}

Finally, we recall some basics from Fourier analysis. Let $\ip{\cdot,\cdot}$ be the standard inner product in $\R^d$, and fix $i \eqdef \sqrt{-1}$. For any functions $f, g: \Z^d \rightarrow \C$ with finite support, the Fourier transform $\widehat{f}: [0,1]^d \rightarrow \C$ is defined by $\widehat{f}(\mathbf{r}) \eqdef \sum_{x \in \Z^d}f(\mathbf{x})e^{-2\pi i\ip{\mathbf{x},\mathbf{r}}}$, and the convolution of $f$ and $g$ is $(f\star g)(\mathbf{x}) \eqdef \sum_{\mathbf{r} \in \Z^d}f(\mathbf{r})g(\mathbf{x} - \mathbf{r})$. Further, the convolution identity states that $\widehat{f\star g} = \widehat{f}\cdot\widehat{g}$ pointwise, and Parseval's identity states that 
\[\sum_{\mathbf{x} \in \Z^d}f(\mathbf{x})\overline{g(\mathbf{x})} = \int_{[0,1]^d}\widehat{f}(\mathbf{r})\overline{\widehat{g}(\mathbf{r})}d\mathbf{r}.\]

The next theorem is the lower bound part of Theorem~\ref{thm:generalbodiesthm}.
\begin{theorem}[Convex Body Lower Bound]\label{thm:generalbodies-lb}
    For any convex body $K \subseteq \R^d$ and $\mathbf{v} \in \R^d$, define $\AA_{\mathbf{v} + K}$ to be the family of all arithmetic progressions restricted to $\Omega_{\mathbf{v} + K}$, the set of integer points in $\vec{v} + K$. Then
    \[f\left(K\right) \lesssim_d \sup_{\mathbf{v} \in \R^d} \disc\left(\AA_{\mathbf{v} + K}\right).\]
    Moreover, if $K=-K$, then 
    \(
    f\left(K\right) \lesssim_d \disc\left(\AA_{K}\right).
    \)
\end{theorem}
\begin{proof}

    We choose $\vec{v}$ to achieve $\max_{v \in R^d} |\Z^d \cap (\vec{v} + K)|$.
    Take a coloring $\chi:\Omega_{\vec{v} + K} \to \{-1,+1\}$, and extend it to a function $\chi: \Z^{d} \rightarrow \{-1, 0, 1\}$ by setting $\chi(\mathbf{x}) = 0$ for all $\mathbf{x} \notin \Omega_{\mathbf{v} + K}$. Further, for each $\mathbf{b} \in \Z^d\setminus \{\vec{0}\}$, define the ``comb'' function $g_{\mathbf{b}}: \Z^{d}\rightarrow \{0, 1\}$ as the indicator function of the set $\{\mathbf{0}, \mathbf{b}, ..., (\ell - 1)\mathbf{b}\}$ for a positive integer $\ell$ to be determined. Note that the convolution $g_{\textbf{b}}\star \chi(\mathbf{x})$ is exactly the discrepancy of an arithmetic progression with step size $-\mathbf{b}$ whose first point is $\mathbf{x}$, i.e.,
    \[
    g_{\mathbf{b}}\star\chi(\mathbf{x}) 
    = \chi\left(\Omega_{\mathbf{v} + K} \cap \{\mathbf{x} - t\mathbf{b}: 0 \leq t < \ell\}\right)
    =\chi(\Omega_{\vec{v} + K} \cap \AP_d(\vec{x},-\vec{b},\ell)),
    \]
    where we used the convention $\chi(S) \eqdef \sum_{\vec{z} \in S}\chi(\vec{z})$ for a set $S \subseteq \Z^d$.
    Let $\mathbf{b} \in m(K - K)\setminus \{\vec{0}\}$ for some real value $m$ to be determined, and define $\widetilde{K} \eqdef K + m\ell(K-K)$. We have
    \begin{equation}\label{eq:generallb-disc-inequality}
         \sum_{\mathbf{x} \in \Z^d}|g_{\mathbf{b}}\star\chi(\mathbf{x})|^2 \leq \disc\left(\AA_{\mathbf{v} + K}\right)^2|\Z^d \cap (\vec{v} + \widetilde{K})|,
    \end{equation}
    since there are at most $|\Z^d \cap (\vec{v} + \widetilde{K})|$ values of $\mathbf{x}$ for which $g_{\mathbf{b}}\star\chi(\mathbf{x})$ is non-zero, and each such value is the discrepancy of an arithmetic progression in $\AA_{\vec{v}+K}$. To see this, recall that $\mathbf{b} \in m(K - K)$ and that $\chi$ is nonzero only on $\Omega_{\vec{v} + K}$. Thus every $\mathbf{x}$ for which $g_{\mathbf{b}}\star\chi(\mathbf{x})\neq 0$ must satisfy $\vec{x} - t\vec{b} \in \vec{v} + K$ for some $0 \le t < \ell$, which implies that $\vec{x} \in \vec{v} + K + m\ell(K-K)$. 

    Since $\widetilde{K} - \widetilde{K} = (1+2m\ell)(K-K)$, equation~\eqref{eq:generallb-disc-inequality} and Lemma~\ref{lem:zeta-maxshift} imply that
    \begin{equation}\label{eq:lb-disc-diffK}
        \sum_{\mathbf{x} \in \Z^d}|g_{\mathbf{b}}\star\chi(\mathbf{x})|^2
        \le 
        \disc\left(\AA_{\mathbf{v} + K}\right)^2\zeta_{K-K}\left(1+2m\ell\right).
    \end{equation}
    On the other hand, by Parseval's identity and the convolution identity, we have
    \begin{align*}
        \sum_{\mathbf{x} \in \Z^d}|g_{\mathbf{b}}\star\chi(\mathbf{x})|^2 &= \int_{[0,1]^d}\widehat{g_{\mathbf{b}}\star\chi}(\mathbf{r})\overline{\widehat{g_{\mathbf{b}}\star\chi}(\mathbf{r})}d\mathbf{r}\\ 
        &= \int_{[0,1]^d}\widehat{g_{\mathbf{b}}}(\mathbf{r})\widehat{\chi}(\mathbf{r})\overline{\widehat{g_{\mathbf{r}}}(\mathbf{r})\widehat{\chi}(\mathbf{r})} d\mathbf{r}
        =\int_{[0,1]^d}|\widehat{g_{\mathbf{b}}}(\mathbf{r})|^2|\widehat{\chi}(\mathbf{r})|^2 d\mathbf{r}.
    \end{align*}
    Together with equation~\eqref{eq:lb-disc-diffK}, we obtain  
    \begin{equation}
        \int_{[0,1]^d}|\widehat{g}_{\mathbf{b}}(\mathbf{r})|^2|\widehat{\chi}(\mathbf{r})|^2 d\mathbf{r} \leq \disc\left(\AA_{\mathbf{v} + K}\right)^2\zeta_{K-K}\left(1+2m\ell\right).
    \end{equation}
    Summing over all $\mathbf{b} \in B\eqdef m(K-K)\setminus \{\vec{0}\}$, we have that 
    \begin{equation}\label{eq:generallb-summed}
        \int_{[0,1]^d}\left(\sum_{\mathbf{b} \in B}|\widehat{g_{\mathbf{b}}}(\mathbf{r})|^2\right) |\widehat{\chi}(\mathbf{r})|^2 d\mathbf{r} \leq \disc\left(\AA_{\mathbf{v} + K}\right)^2\zeta_{K-K}\left(1+2m\ell\right)\zeta_{K-K}\left(m\right).
    \end{equation}
    We will now show that the sum $\sum_{\mathbf{b} \in B}|\widehat{g_{\mathbf{b}}}(\mathbf{r})|^2$ can be bounded uniformly from below by a constant multiple of $\ell^2$.
    Let us fix some $\mathbf{r} \in [0, 1]^{d}$. We can find distinct integer points $\mathbf{b}_1, \mathbf{b}_2 \in \frac{m}{2}(K-K)$ such that the fractional parts of $\ip{\mathbf{b}_1,\mathbf{r}}$ and $\ip{\mathbf{b}_2,\mathbf{r}}$ differ by at most $\varepsilon \eqdef \frac{1}{\zeta_{K-K}(m/2)-1}$ by the pigeonhole principle. It follows that there exists some $\mathbf{b}_{\vec{r}} \in B$ such that the fractional part of $\ip{\mathbf{b}_{\vec{r}}, \mathbf{r}}$ lies in $\left[0, \varepsilon\right]$, namely $\vec{b}_{\vec{r}} \eqdef \mathbf{b}_1 - \mathbf{b}_2$. For any 
    \begin{equation}\label{eq:lb-ell-bound}
    \ell \le 1+\frac{1}{6\varepsilon} = \frac{5}{6} + \frac{1}{6}\zeta_{K-K}(m/2),
    \end{equation}
    and any $0 \le t \le \ell-1$, we have
    \[
   1\ge \Re \exp\left(-2\pi i t\ip{\vec{b}_{\vec{r}}, \vec{r}}\right)\ge \cos(2\pi(\ell-1)\varepsilon) \ge \cos\left(\frac{\pi}{3}\right) = \frac{1}{2}.
    \]
    Therefore, we get the bound
    \begin{equation}\label{eq:generallb-gbbound}
        \sum_{\mathbf{b} \in B}|\widehat{g_{\mathbf{b}}}(\mathbf{r})|^2 \geq |\widehat{g_{\mathbf{b}_{\vec{r}}}}(\mathbf{r})|^2 = \left|\sum_{t = 0}^{\ell - 1} \exp\left(- 2\pi i t\ip{\mathbf{b}_{\vec{r}},\mathbf{r}}\right)\right|^2 \ge \frac{\ell^2}{4}.
    \end{equation}
    Since we also have, by Parseval's identity,
    \begin{equation*}
        \int_{[0,1]^d}|\widehat{\chi}(\mathbf{r})|^2 = \sum_{\mathbf{x} \in \Z^d}|\chi(\mathbf{x})|^2 = |\Z^d\cap \mathbf{v} + K|,
    \end{equation*}
    together with equation~\eqref{eq:generallb-summed} and equation~\eqref{eq:generallb-gbbound}, for any $\ell$ satisfying \eqref{eq:lb-ell-bound} we have the inequality
    \begin{equation}\label{eq:generallb-finalbound}
        \disc\left(\AA_{\mathbf{v} + K}\right)^2\zeta_{K-K}\left(1+2m\ell\right)\zeta_{K-K}\left(m\right) \ge \frac{1}{4}\ell^2|\Z^d\cap \mathbf{v} + K|.
    \end{equation}
    We can now choose values for the parameters $\ell$ and $m$.
    Let $\ell \eqdef \frac{f(K)^2}{12}$ and $m \eqdef \frac{4}{f(K)^2}$. First we show that $\ell \leq \frac{\zeta_{K-K}(m/2)}{6}$, which implies that equation~\eqref{eq:lb-ell-bound} holds. Since $f(K)^2$ is defined to be the \emph{smallest} $s$ such that $\zeta_{K-K}\left(\frac{1}{s}\right) \leq s$ (see equation~\eqref{eq:ub-body}), values of $s$ less than $f(K)^2$, namely $\frac{f(K)^2}{2} = \frac{2}{m}$, must satisfy 
    \[\zeta_{K-K}\left(m/2\right) > \frac{2}{m} =  \frac{f(K)^2}{2}.\]
    It follows that $\frac{\zeta_{K-K}(m/2)}{6} > \frac{f(K)^2}{12} = \ell$ as required. 
    
    Substituting $\ell = \frac{f(K)^2}{12}$ and $m = \frac{4}{f(K)^2}$ into equation~\eqref{eq:generallb-finalbound}, we have
    \begin{align}
        \disc\left(\AA_{\mathbf{v} + K}\right)^2\zeta_{K-K}\left(\frac{5}{3}\right)\zeta_{K-K}\left(\frac{4}{f(K)^2}\right) \gtrsim f(K)^4|\Z^d\cap \mathbf{v} + K|.\label{eq:lb-substituted}
    \end{align}
    By Lemmas~\ref{lem:zeta-maxshift}~and~\ref{lem:zeta-scaling},
    \begin{equation}\label{eq:lb-const-scale}
        \zeta_{K-K}\left(\frac53\right) \lesssim_d \zeta_{K-K}(1) \lesssim_d |\Z^d\cap \mathbf{v} + K|,
    \end{equation}
    where, in the last inequality, we used both Lemma~\ref{lem:zeta-maxshift} and the choice of $\vec{v}$ to maximize $|\Z^d \cap (\vec{v} + K)|$. Furthermore, by Lemma~\ref{lem:zeta-scaling}, and the definition of $f(K)$, we have
    \begin{equation}\label{eq:lb-fK}
        \zeta_{K-K}\left(\frac{4}{f(K)^2}\right) \lesssim_d \zeta_{K-K}\left(\frac{1}{f(K)^2}\right) \le f(K)^2.
    \end{equation} 
    Together, the inequalities \eqref{eq:lb-substituted}, \eqref{eq:lb-const-scale}, and \eqref{eq:lb-fK} give us $\disc(\AA_{\vec{v}+K})^2 \gtrsim_d f(K)^2$, as we needed to prove.

    To see that the stronger lower bound for centrally symmetric convex bodies $K$ (ones that satisfy $K=-K$) holds, observe that the only place where we used the choice of $\vec{v}$ was in establishing the last inequality in \eqref{eq:lb-const-scale}. All other inequalities we derived hold for an arbitrary $\vec{v}$, including $\vec{v} = \vec{0}$.  For centrally symmetric convex bodies, equation~\eqref{eq:lb-const-scale} instead becomes
    \[
    \zeta_{K}\left(\frac53\right) \lesssim_d \zeta_{K}(1)
    \]
    using the symmetry of $K$ and Lemma~\ref{lem:zeta-scaling}. The rest of the proof is identical.
\end{proof}

Similarly to our discussion of $\gamma_2(\AA_{\vec{N}})$ after the proof of Lemma~\ref{lem:gamma2-ap-ub}, the proof of Theorem~\ref{thm:generalbodies-lb} also shows the lower bound $\sup_{\vec{v} \in \R^d} \gamma_2(\AA_{\vec{v} + K}) \gtrsim_d f(K)$ for convex bodies $K$, and also $\gamma_2(\AA_{K}) \gtrsim_d f(K)$ when $K$ is centrally symmetric. At present we do not have a matching upper bound on the $\gamma_2$ norm. Proving such an upper bound would also imply a constructive version of Theorem~\ref{thm:generalbodies-ub} (i.e., the upper bound part of Theorem~\ref{thm:generalbodiesthm}).


\printbibliography[heading=bibintoc]
\end{document}